\documentclass{amsart}

\usepackage{url}
\usepackage[a4paper, total={5.5in, 8in}]{geometry}
\usepackage{lscape}
\usepackage{mathtools}
\usepackage[colorlinks=true,linkcolor=red,citecolor=blue]{hyperref}
\usepackage{color}
\usepackage[utf8]{inputenc}
\usepackage[color,matrix,arrow, curve]{xy}
\usepackage{amsmath, amsthm, amsfonts, amssymb}
\usepackage{mathrsfs}
\usepackage{tikz}
\usetikzlibrary{matrix,arrows}

\DeclareMathOperator{\Hom}{Hom}

\DeclareMathOperator{\map}{Map}

\renewcommand{\lim}{\mathrm{lim}}

\numberwithin{equation}{section}
\newtheorem{thm}{Theorem}[section]
\newtheorem*{thma}{Theorem A}
\newtheorem*{thmb}{Theorem B}

\newtheorem*{thmd}{Theorem D}
\newtheorem*{corc}{Corollary C}
\newtheorem{cor}[thm]{Corollary}
\newtheorem{lem}[thm]{Lemma}

\newtheorem{prop}[thm]{Proposition}

\theoremstyle{definition}
\newtheorem{defn}[thm]{Definition}

\newtheorem{rem}[thm]{Remark}

\newcommand{\id}{\textup{id}}

\newcommand{\holim}{\textup{holim}}

\newcommand{\hofiber}{\textup{hofiber}}

\newcommand{\emb}{\textup{emb}}

\newcommand{\imm}{\textup{imm}}

\newcommand{\sO}{\mathcal O}

\newcommand{\op}{\textup{op}}

\newcommand{\FF}{\mathbb F}
\newcommand{\RR}{\mathbb R}
\newcommand{\ZZ}{\mathbb Z}

\newcommand{\QQ}{\mathbb Q}

\newcommand{\holimsub}[1]{\begin{array}[t]{cc} \textup{holim} \\ [-1mm]
\scriptstyle{#1} \end{array}}

\newcommand{\uli}{\underline}

\newcommand{\cat}{\mathbf}

\begin{document}

\title{Galois symmetries of knot spaces}
\author[Pedro Boavida]{Pedro Boavida de Brito}
\email{pedrobbrito@tecnico.ulisboa.pt}
\address{Dept. of Mathematics, IST, Univ. of Lisbon, Av. Rovisco Pais, Lisboa, Portugal}%

\author{Geoffroy Horel}
\email{horel@math.univ-paris13.fr}
\address{Universit\'e Sorbonne Paris Nord, LAGA, CNRS, UMR 7539, F-93430, Villetaneuse, France \newline \'Ecole normale sup\'erieure, DMA, CNRS, UMR 8553, 45 rue d'Ulm, 75230 Paris Cedex 05, France}

\begin{abstract}
We exploit the Galois symmetries of the little disks operads to show that many differentials in the Goodwillie-Weiss spectral sequences approximating the homology and homotopy of knot spaces vanish at a prime $p$. Combined with recent results on the relationship between embedding calculus and finite-type theory, we deduce that the $(n+1)$-st Goodwillie-Weiss approximation is a $p$-local universal Vassiliev invariant of degree $\leq n$ for every $n \leq p + 1$.
\end{abstract}
\thanks{We gratefully acknowledge the support through: grant SFRH/BPD/99841/2014 and project MAT-PUR/31089/2017, funded by Funda\c c\~{a}o para a Ci\^{e}ncia e Tecnologia; projects ANR-14-CE25-0008 SAT, ANR-16-CE40-0003 ChroK, ANR-18-CE40-0017 PerGAMo, funded by Agence Nationale pour la Recherche.}
\maketitle

\vspace*{15pt}
A long \emph{knot} is a smooth embedding $\RR \to \RR^3$ which coincides with the standard inclusion $t \mapsto (t,0,0)$ in the complement of $[0,1]$. The space of such, with the weak Whitney topology, is denoted $K$. In this paper, we show that the Goodwillie-Weiss approximation to the space of knots gives a $p$-local analog of Kontsevich's integral, in a range which improves as the prime $p$ increases. 

\begin{thma}\label{thm:univ}
Let $p$ be a prime and $n \leq p +1$. The $(n+1)$-st Goodwillie-Weiss approximation of the space of knots
\[K\to T_{n+1}(K)\]
is, on path components, a universal finite type invariant of degree $n$ over $\ZZ_{(p)}$, the localization of $\ZZ$ at the prime $p$.

Furthermore, under the same assumption on $n$, there is a non-canonical isomorphism 
\[\pi_0(T_{n+1}(K)) \otimes \ZZ_{(p)} \cong \oplus_{s\leq n}\mathcal{A}_s^I \otimes \ZZ_{(p)} \]
where $\mathcal{A}_s^I$ is the algebra of indecomposable Feynman diagrams, that is, the quotient of the free abelian group on unitrivalent trees of degree $s$ by the antisymmetry, $IHX$ and $STU^2$ relations.
\end{thma}

We now describe in more detail the main characters of this story: the \emph{Vassiliev tower} and the \emph{Goodwillie-Weiss tower}. 

\medskip
 A finite-type invariant of degree $n$ is a homomorphism from $\pi_0(K)$ (viewed as a monoid with respect to concatenation) to an abelian group, whose extension to singular knots having $n+1$ double points vanishes. Following Vassiliev, Goussarov and Birman-Lin, the \emph{extension} $\overline{\mu}$ of a knot invariant $\mu$ to singular knots $K$ having at most $k$ double points is defined inductively by the formula: $\overline{\mu}(K) = \overline{\mu}(K_+) - \overline{\mu}(K_-)$ where $K_+$ and $K_-$ are the knots obtained from $K$ by resolving a given double point in the two possible ways. Some standard references are \cite{Vassiliev}, \cite{BarNatan}. 
 
 It is not known whether finite type invariants -- also called Vassiliev invariants -- distinguish all knots, but it is known that they abound. Examples include all so-called quantum invariants, and so all invariants coming from perturbative Chern-Simons theory, as well as many classically studied knot invariants. Two knots are declared to be $n$-equivalent if all finite type invariants of degree $\leq n$ agree on these knots; this defines an equivalence relation $\sim_n$ on $\pi_0(K)$. The equivalence relation $\sim_n$ is finer than $\sim_{n-1}$, and therefore we have a tower of surjections
\begin{equation}\label{Vassiliev}
\pi_0(K)\to\ldots\to\pi_0(K)/\!\!\sim_n\to \pi_0(K)/\!\!\sim_{n-1}\to\ldots\to\pi_0(K)/\!\!\sim_1
\end{equation}
that we call the \emph{Vassiliev tower}. It turns out \cite[Theorem 3.1]{Gusarov} that the concatenation of knots endows each of the sets $\pi_0(K)/\!\!\sim_n$ with the structure of a finitely generated abelian group. There is moreover a surjective map 
\[\mathcal{A}_n^I\to\mathrm{ker}[\pi_0(K)/\!\!\sim_n\to \pi_0(K)/\!\!\sim_{n-1}]\;.\]
A deep and fundamental result in the development of the subject was Kontsevich's construction of the so-called Kontsevich integral which produces a rational inverse of this map and so a (non-canonical) isomorphism
\[(\pi_0(K)/\!\!\sim_n)\otimes\QQ\cong\oplus_{s\leq n}\mathcal{A}_s^I\otimes\QQ \; .\]

\medskip
The \emph{Goodwillie-Weiss tower}, on the other hand, is a tower of topological spaces
\[K\to\dots\to T_n(K)\to T_{n-1}(K)\to\dots\to T_2(K)\]
The map $K \to T_n (K)$ records the value of an embedding $\RR\to\RR^3$ on subsets of $\RR$ given by the complement of at most $n+1$  points of $(0,1)$ and is, in a homotopical sense, universal with respect to this data. One can hit the Goodwillie-Weiss tower with $\pi_0$ and obtain a tower 
\begin{equation}\label{GW}
\pi_0(K)\to\dots\to \pi_0(T_n(K))\to \pi_0(T_{n-1}(K))\to\dots\to \pi_0(T_2(K)) \; .
\end{equation}

It is known that all the sets $\pi_0(T_n(K))$ are finitely generated abelian groups and that all the maps in the tower are compatible with this structure \cite{BCKS,BoavidaWeiss}. There is a spectral sequence that computes $\pi_0(T_n(K))$ (and also its higher homotopy groups): 
\[E^1_{-s,t}\implies \pi_{t-s}T_n(K) \; .\]
The $E^1$-page consists of the homotopy groups of the homotopy fibers of the maps $T_s(K)\to T_{s-1}(K)$ with $s\leq n$ and can be expressed in terms of homotopy groups of spheres. The $d^1$-differential can be described explicitly and we have an isomorphism  
\[E^2_{-s,s}\cong \mathcal{A}_{s-1}^I\;\]
(see \cite{conanthomotopy} or \cite{Shi} for an alternative account).

\medskip
A long-standing conjecture is that the towers (\ref{Vassiliev}) and (\ref{GW}) are isomorphic; see the last paragraph of \cite{GoodwillieWeiss} and \cite[Conjecture 1.1]{BCSS}. Some important steps have been made in this direction. Volić thesis \cite{volic} provided the first indication, though in the homological context. In \cite{BCKS}, Budney-Conant-Koytcheff-Sinha showed that the map $\pi_0(K)\to\pi_0(T_n(K))$ is a degree $(n-1)$ invariant. A new proof of this result, using grope cobordisms, will appear in \cite{KST}. The upshot is that the tower (\ref{Vassiliev}) maps to the tower (\ref{GW}) (with a shift of degree), that is, we have a commutative ladder of abelian groups
\[
\xymatrix{
\ldots\ar[r]&\pi_0(K)/\!\!\sim_n\ar[r]\ar[d]&\pi_0(K)/\!\!\sim_{n-1}\ar[r]\ar[d]&\ldots\ar[r]&\pi_0(K)/\!\!\sim_1\ar[d]\\
\ldots\ar[r]&\pi_0(T_{n+1}(K))\ar[r]&\pi_0(T_n(K))\ar[r]&\ldots\ar[r]&\pi_0(T_2(K)) \; .\\
}
\]

By work of Conant, Kosanović, Teichner and Shi \cite{CTgrope,Kosanovic,KST}, a sufficient condition for the vertical maps to be isomorphisms is that the Goodwillie-Weiss spectral sequence collapses along the anti-diagonal (Theorem \ref{thm:collapse implies universal}). Theorem A is then a consequence of such collapse results. To put it differently, Theorem A says that the kernel of the vertical map $\pi_0(K)/\!\!\sim_n\to\pi_0(T_{n+1}(K))$ can only have $p$-power torsion for $p< n-1$. This result is modest in comparison to the conjecture that the two towers are isomorphic over $\ZZ$. Nevertheless, it seems to be the first non-rational computation about the Goodwillie-Weiss tower beyond low degrees ($n \leq 3$ \cite{BCSS}).

We also note that, after tensoring with $\QQ$, our result gives a construction of an isomorphism
\[(\pi_0(K)/\!\!\sim_{n})\otimes\QQ\cong \oplus_{s \leq n}\mathcal{A}^I_s\otimes\QQ\] 
that is purely homotopical and does not rely on transcendental techniques like Kontsevich's integral. However, our splitting is not explicit, whereas Kontsevich's integral gives a preferred choice of splitting over the complex numbers.

\medskip
Theorem A is a corollary of more general results about the Goodwillie-Weiss tower for knots in $\RR^d$-- namely Theorem B below -- as we now explain. This tower converges to the space of knots when $d \geq 4$, in the sense that the map from $\emb_c(\RR, \RR^d)$ to the homotopy limit of the tower is a weak equivalence. For $d = 3$, the case of classical knots, the tower is known not to converge to the knot space. (For example, $\pi_0 \emb_c(\RR, \RR^3)$ is countable whereas $\pi_0T_\infty \emb_c(\RR, \RR^3)$ is uncountable.)  

The spectral sequence associated to the Goodwillie-Weiss tower has $E^1$-page consisting of the homotopy groups of the layers of the tower, just like in the case $d = 3$. \emph{Rationally}, this spectral sequence is well understood: from the work of Arone, Lambrechts, Voli\'{c} and Turchin  \cite{ALTV} we know that there is rational collapse at the $E^2$-page when $d \geq 4$. Our main technical result gives a vanishing range of differentials for the spectral sequence associated to Goodwillie-Weiss tower (or to a truncation of it), at a prime $p$:
\begin{thmb}
Let $p$ be a prime number. The differential
\[
d^r_{-s,t} : E^{r}_{-s,t} \to E^{r}_{-s-r, t+r-1}
\]
vanishes $p$-locally if $r-1$ is not a multiple of $(p-1)(d-2)$ and
\[
t < 2p - 2 + (s-1)(d-2) \; .
\]
\end{thmb}

Together with a result about triviality of extensions, this gives a computation of some homotopy groups of $T_{n} \emb_c(\RR, \RR^d)$ for $d \geq 3$. The range improves as the prime or the codimension increases. The precise result is Theorem \ref{thmC}. Here we record an application on the homotopy groups of knot spaces for $d \geq 4$, at a prime $p$, which uses Goodwillie-Klein's excision estimates:

\begin{corc}\label{corc}
For $p$ a prime, there is an isomorphism
\[
\pi_i \emb_c(\RR^1, \RR^d) \otimes \ZZ_{(p)} \cong \oplus_{t-s = i} E^2_{-s,t}\otimes\ZZ_{(p)}
\]
for $i < 2p$ and $d = 4$ and for $i < 2p+2d-4$ and $d > 4$.
\end{corc}

\medskip
The proof of Theorem B stands on the relation between the Goodwillie-Weiss tower and the little disks operad  \cite{sinhaoperads,turchindelooping,dwyerhess,BoavidaWeiss}, and the existence of Galois symmetries on the little disks operad that we established in \cite{BHformality}. Our main message in this paper is that by combining these two observations one obtains an interesting Galois action on the Goodwillie-Weiss tower. The numerology in Theorem B may seem esoteric at first, but it is a consequence of the action: the inequality has something to do with the range over which the Galois action factors through the cyclotomic character.

\medskip
We also prove a homological statement in Section \ref{sec:homology} which, perhaps unsurprisingly, is stronger. It is also easier to prove, even though the strategy is the same. In \cite{sinhaoperads}, Sinha constructs a cosimplicial space $K^\bullet$ which is a model for the Goodwillie-Weiss tower. Associated to the said cosimplicial space, and given a ring $R$ of coefficients, there is a homological Bousfield-Kan spectral sequence $E^*(R)$. Modulo convergence issues, this spectral sequence is related to $H_*( \overline{\emb}_c(\RR^1, \RR^d), R)$ where $\overline{\emb}_c(\RR^1, \RR^d)$ denotes the homotopy fiber of the map
\[\emb_c(\RR^1, \RR^d)\to \imm_c(\RR^1, \RR^d)\;.\]
Our main theorem about the homology spectral sequence is the following.

\begin{thmd}
Let $p$ be a prime number. The only possibly non-trivial differentials in the spectral sequence $E^*(\ZZ_p)$ are $d_{1+n(d-1)(p-1)}$ for $n \geq 0$.
\end{thmd}

This theorem was announced in \cite{HorelSMF}. Asymptotically, we also recover the collapse results for $R = \QQ$, first established by Lambrechts-Turchin-Voli\'{c} in \cite{lambrechtsrational} using the formality of the little disks operad.

\subsection*{Acknowledgments} 
We wish to thank Peter Teichner, Yuqing Shi and Danica Kosanović for valuable discussions. 

\section{Bousfield localization}

In this section, we recall some facts about localization of spaces. We denote the $p$-adic integers by $\ZZ_p$. Let $E$ be a homology theory on spaces.

\begin{defn}
A space $X$ is $E$-local if for any map $f:U\to V$ such that $E(f)$ is an isomorphism, the map induced by precomposition
\[\map(V,X)\to\map(U,X)\]
is a weak equivalence.
\end{defn}

\begin{thm}[Bousfield \cite{bousfieldlocalization}]
There exists an endofunctor $L_E$ of the category of spaces equipped with a natural transformation $\id\to L_E$ satisfying the following two conditions.
\begin{enumerate}
\item For each $X$ the space $L_E(X)$ is $E$-local.
\item The map $X\to L_E(X)$ induces an isomorphism in $E$-homology.
\end{enumerate}
Moreover, the functor $L_E$ is uniquely defined up to weak equivalence by these two properties.
\end{thm}

In the case where $E$ is the homology theory $H_*(-,\FF_p)$, the functor $L_E$ will be denoted $L_p$ and called $p$-completion, following common practice. This is justified by the next proposition.

\begin{prop}
If $X$ is a simply connected space whose homotopy groups are finitely generated, the map $X\to L_pX$ induces the canonical map $\pi_*(X)\to \pi_*(X)\otimes\ZZ_p$ on homotopy groups. 
\end{prop}

\begin{proof}
See \cite[Proposition 4.3 ($ii$)]{bousfieldlocalization}.
\end{proof}
We will need two  facts about this construction.

\begin{prop}\label{prop: completion preserves products}
The $p$-completion functor $L_p$ commutes with products up to homotopy. More precisely, the canonical map
\[L_p(X\times Y)\to L_p(X)\times L_p(Y)\]
is an equivalence.
\end{prop}

\begin{proof}
The cannonical map $X\times Y\to L_p(X)\times L_p(Y)$ is a mod $p$ homology equivalence by the Künneth isomorphism. Moreover, the target is easily seen to be $p$-complete. Therefore, by uniqueness of the $p$-completion, $L_p(X)\times L_p(Y)$ must be the $p$-completion of $X\times Y$.
\end{proof}

The following proposition shows that under some mild assumptions on $X$, the homology of $X$ with coefficients in the $p$-adic integers $\ZZ_p$ can be recovered from the $p$-completion. We denote by $C_*$ the standard singular chains functor from spaces to chain complexes.

\begin{prop}\label{prop: p adics}
Let $X$ be a space whose homology with coefficients in $\ZZ_p$ is degreewise finitely generated. Then the composite
\[C_*(X,\ZZ_p)\to \lim_n C_*(X,\ZZ/p^n)\to \lim_n C_*(L_pX,\ZZ/p^n)\]
is a quasi-isomorphism.
\end{prop}

\begin{proof}
We show that each of the two maps is a quasi-isomorphism. First, we observe that each of the two towers is a tower of epimorphisms. It follows that the limit agrees with the homotopy limit. In order to prove that the second map is a quasi-isomorphism, it is enough to prove that the map $f:X\to L_p X$ is a mod $p^n$ homology equivalence for all $n$. This follows by an easy induction using the Bockstein long exact sequence
\[\ldots\to H_i(-,\ZZ/p^n)\to H_i(-,\ZZ/p^{n+1})\to H_i(-,\ZZ/p)\to H_{i-1}(-,\ZZ/p^n)\to\ldots \]
We now show that the map
\[C_*(X,\ZZ_p)\to \lim_n C_*(X,\ZZ/p^n)\]
is a quasi-isomorphism. The homology groups of the source are, by definition, the homology groups of $X$ with $\ZZ_p$ coefficients and the homology groups of the target are
\[H_i(\lim_n C_*(X,\ZZ/p^n))\cong \lim_nH_i(X,\ZZ/p^n).\]
This isomorphism follows from a Milnor short exact sequence argument (the assumption made on $X$ implies that the mod $p^n$ homology is degreewise finite and thus the $\lim^1$-term vanishes by the Mittag-Leffler criterion). Thus we have to prove that the canonical map
\[H_i(X,\ZZ_p)\to  \lim_n H_i(X,\ZZ/p^n)\]
is an isomorphism. This map can be factored as
\[H_i(X,\ZZ_p)\to\lim_nH_i(X,\ZZ_p)\otimes\ZZ/p^n\to  \lim_n H_i(X,\ZZ/p^n).\]
The first of these two maps is an isomorphism under the finite type assumption. The second map is also an isomorphism. This can be seen as follows: using the universal coefficients theorem and the left exactness of the limit functor, we have an exact sequence
\[0\to \lim_nH_i(X,\ZZ_p)\otimes\ZZ/p^n\to\lim_n H_i(X,\ZZ/p^n)\to \lim_n\mathrm{Tor}^{\ZZ_p}(\ZZ/p^n,H_{i-1}(X,\ZZ_p)).\]
But, for any finitely generated $\ZZ_p$-module $A$, the group $\lim_n\mathrm{Tor}(\ZZ/p^n,A)$ is zero.
\end{proof}

\begin{rem}
This proposition may be false if the homology of $X$ is not finitely generated. As a counter-example, one can take $X$ to be the rationalization of $S^1$. On the one hand, its $p$-completion is contractible. On the other hand, $H_1(X,\ZZ_p)$ is isomorphic to $\QQ_p$.
\end{rem}

\section{A Grothendieck-Teichm\"uller action on the little disks operads}

The Grothendieck-Teichm\"uller group $GT_p$ introduced in \cite{Drinfeldquasi} is a profinite group equipped with a surjective map
\[\chi:GT_p\to \ZZ_p^\times\]
called the cyclotomic character. Given a $\ZZ_p$-module $M$, we can give it a $GT_p$-action using the formula
\[\gamma.m:=\chi(\gamma)^nm\;.\]
We call this action the weight $n$ cyclotomic action.

\begin{prop}\label{prop: submodule and quotient}
Let $M$ $M'$ and $M''$ be $\ZZ_p$-modules equipped with a $GT_p$-action. Suppose that there is a short exact sequence of $GT_p$-equivariant homomorphisms of $\ZZ_p$-modules
\[0\to M'\to M\to M''\to 0 \; .\]
If the $GT_p$-action on $M$ is cyclotomic of weight $n$, then that is also the case for the $GT_p$-action on $M'$ and $M''$.
\end{prop}

\begin{proof}
This is an immediate verification.
\end{proof}

\begin{prop}\label{prop: no map}
Let $M$ and $N$ be two $\ZZ_p$-modules equipped, respectively, with the weight $m$ and weight $n$ cyclotomic $GT_p$-action. If $m-n$ is not a multiple of $p-1$, any $GT_p$-equivariant homomorphism of $\ZZ_{p}$-modules
\[f:M\to N\]
must be the zero map.

\end{prop}

\begin{proof}
Let us assume that $m-n$ is not a multiple of $p-1$. Let $x$ be any element of $M$. If $f$ is a $GT_p$-equivariant map, we must have 
\[\chi(\gamma)^mf(x)=\chi(\gamma)^nf(x) \; \]
for every $\gamma$. In particular, since $\chi$ is surjective, for any unit $u$ of $\ZZ_p$ we have that $(u^{n-m}-1)\cdot f(x) = 0$. Hence, it is enough to find $u$ such that $u^{n-m}-1$ is also a unit. Take a generator $v$ of $\FF_p^\times$. Then $v^{n-m} \neq 1$ if $n-m$ is not a multiple of $p-1$ (since $\FF_p^\times$ is cyclic of order $p-1$). Let $u$ be a lift of $v$. Then $u^{n-m} -1$ is congruent to $v^{n-m} -1$ modulo $p$ and so it is non-zero modulo $p$, i.e. it is a unit.
\end{proof}

\begin{rem}\label{rem : Zp linear}
In applying these two propositions, $\ZZ_p$-linearity often comes for free.
Namely, if $M$ and $N$ are finitely generated $\ZZ_p$-modules, then any morphism of abelian groups $f:M\to N$ is automatically $\ZZ_p$-linear. Indeed, we have the following sequence of isomorphisms:
\begin{align*}
\Hom_{\ZZ_p}(M,N)&\cong\mathrm{lim}_n\Hom_{\ZZ_p}(M,N\otimes_{\ZZ_p}\ZZ/p^n)\\
   &\cong\mathrm{lim}_n\Hom_{\ZZ/p^n}(M\otimes_{\ZZ_p}\ZZ/p^n,N\otimes_{\ZZ_p}\ZZ/p^n)\\
      &\cong\mathrm{lim}_n\Hom_{\ZZ/p^n}(M\otimes_{\ZZ}\ZZ/p^n,N\otimes_{\ZZ_p}\ZZ/p^n)\\
   &\cong\mathrm{lim}_n\Hom_{\ZZ}(M,N\otimes_{\ZZ_p}\ZZ/p^n)\\
   &\cong\Hom_{\ZZ}(M,N)
\end{align*}
where the first and last isomorphisms hold because finitely generated $\ZZ_p$-modules are $p$-complete.
\end{rem}

The next theorem plays a key role in what follows; it equips the $p$-completion of the little disks operad with an action of $GT_p$. We denote by $E_d$ the little disks operad of dimension $d$. We view $E_d$ as an $\infty$-operad (a dendroidal space satisfying the Segal condition) via its nerve. Then, applying $L_p$ objectwise, we obtain a dendroidal space $L_pE_d$ which, by Proposition \ref{prop: completion preserves products}, also satisfies the Segal condition. It is important to view $E_d$ as an $\infty$-operad since the $p$-completion functor $L_p$ does not preserve strict products, so applying $p$-completion aritywise to an operad does not directly produce an operad, but rather an $\infty$-operad. For more details, see \cite{BHformality}.

\begin{thm}\label{thm:GtonLpE}
Let $d$ be an integer with $d\geq 3$. There exists a $GT_p$-action on $L_pE_d$ such that the induced action on $H_{d-1}(L_pE_{d}(2),\ZZ_p)$ is the cyclotomic action of weight $1$. Moreover, if we give $E_1$ the trivial $GT_p$-action, the canonical map $E_1\to L_pE_d$ is $GT_p$-equivariant.
\end{thm}
\begin{proof}
This is done in \cite{BHformality}. We recall briefly how this construction goes. For a space $X$, we denote by $X^{\wedge}$ its pro-$p$-completion. This is a pro-object in $p$-finite spaces (a space is $p$-finite if it has finitely many components and if each component is truncated with finite $p$-groups as homotopy groups) that receives a map from $X$ and is initial for this property. There is a close relationship between this functor and the $p$-completion functor $X\mapsto L_p(X)$: when $X$ is a nilpotent space of finite $\FF_p$-type (i.e. all its $\FF_p$-homology groups are finite dimensional) the $p$-completion $L_p(X)$ coincides with the homotopy limit of the pro-object $X^{\wedge}$. This homotopy limit is denoted $\RR\mathrm{Mat}$ in \cite{BHformality} and is the homotopy right adjoint of the pro-$p$-completion functor.

By work of Drinfel'd, there exists a non-trivial action of $GT_p$ on $E_2^{\wedge}$. Using a version of the Dunn-Lurie additivity theorem established in \cite{BWproduct} one can produce an action of $GT_p$ on $E_d^{\wedge}$ by writing it as a certain tensor product of $(E_2)^{\wedge}_{\FF_p}$ with the Drinfel'd action and $E_{d-2}^{\wedge}$ with the trivial action. Morever, with respect to this action, the map $E_1^{\wedge} \to E_d^{\wedge}$ is $GT_p$-equivariant.  Applying the functor $\RR\mathrm{Mat}$, we obtain the asserted $GT_p$-equivariant map $E_1\simeq L_pE_1\to L_pE_d$ using the fact that, for $d\geq 3$, the spaces $E_d(k)$ are nilpotent and of finite type. 

The statement that the action is cyclotomic of weight $1$ on $H_{d-1}(L_pE_{d}(2),\ZZ_p)$ follows from \cite[Proof of Theorem 7.1]{BHformality} where it is shown that the composite
\[GT_p\to\mathrm{Aut}^h(L_pE_d)\to\mathrm{Aut}^h(L_pE_d(2))\cong\ZZ_p^\times\]
coincides with the cyclotomic character.
\end{proof}

From this, we can fully understand the induced $GT_p$-action on homology.

\begin{prop}\label{prop: weight}
For any $k$, the $GT_p$-action on $H_{n(d-1)}(E_d(k),\ZZ_p)$ is cyclotomic of weight $n$.
\end{prop}

\begin{proof}
The homology of $E_d$ is the operad of $d$-Poisson algebras; it is generated by a commutative product of degree $0$, a Lie bracket of degree $(d-1)$, the latter being a derivation with respect to the former in both variables. Since this operad is generated by operations of arity $2$, it follows that the homology $H_*(E_d(k),\ZZ_p)$ is a quotient of tensor powers of $H_*(E_d(2),\ZZ_p)$. The statement of the proposition holds for $k=2$ by Theorem \ref{thm:GtonLpE} and the general case follows from Proposition \ref{prop: submodule and quotient}.
\end{proof}

\section{The Goodwillie-Weiss tower and $p$-completion}

\subsection{Generalities about $p$-completion of towers}
For a tower of based spaces and basepoint preserving maps,
\[
\dots \to T_n \to T_{n-1} \to \dots \to T_0
\]
there is an associated spectral sequence \cite[IX.4]{BousfieldKan} whose $E^1$ page is
\[
E^1_{-s,t}(T) = \pi_{t-s} \hofiber (T_s \to T_{s-1}) \;
\]
and the first differential $d^1 : E^1_{-s,t}(T) \to E^1_{-(s+1),t}(T)$ is the composite
\[
\pi_{t-s} \hofiber (T_s \to T_{s-1}) \to \pi_{t-s} T_s \to \pi_{t-s-1} \hofiber (T_{s+1} \to T_{s})
\]
where the second map is the boundary homomorphism. The differential $d^r$ has degree $(-r, r-1)$.

We assume that each $T_i$ has abelian homotopy groups (including $\pi_0$) and that the induced maps on homotopy groups are group homomorphisms (e.g. a tower of $2$-fold loop spaces). This is the case in our examples below, and having this hypothesis avoids subtleties which arise when defining this spectral sequence in general, when $\pi_0$ and $\pi_1$ are not abelian groups.

Let us briefly discuss convergence. For details, see \cite[IX]{BousfieldKan}. Let $F^s\pi_*(\mathrm{holim}T)$ be the image of the map $\pi_*(\mathrm{holim}T)\to \pi_* T_{s}$. Using the Milnor short exact sequence, we see that, up to a $\lim^1$-term, the group $\pi_* (\mathrm{holim}T)$ is the limit of the tower of epimorphisms
\begin{equation}\label{eq:towerFs}
\dots \to F^s\pi_*(\mathrm{holim}T) \to F^{s-1}\pi_*(\mathrm{holim}T) \to \dots \to F^{0}\pi_*(\mathrm{holim}T)
\end{equation}
(indeed the limit of this tower can easily be identified with $\lim_s \pi_* T_s$) and there are inclusions 
\[
\ker[ F^s\pi_*(\mathrm{holim}T) \to F^{s-1}\pi_*(\mathrm{holim}T) ] \subset E^{\infty}_{-s, s+*}(T) \; .
\]
For these inclusions to be isomorphisms, a sufficient condition is that 
\[
\lim^1_r E^{r}_{-s,t}
\]
vanishes \cite[IX.5.4]{BousfieldKan}. Under this hypothesis, the term $\lim^1_{s}\pi_* T_s$ also vanishes, and so the spectral sequence converges in the sense that the limit of the tower (\ref{eq:towerFs}) is $\pi_* (\mathrm{holim}T)$ and the associated graded of this tower is given by the $E^\infty$-page. The vanishing of $\lim^1_r E^{r}_{-s,t}$ happens, for example, if the connectivity of the layers of the tower $T$ goes to $\infty$ as $n \to \infty$. 

\medskip
Before we get to our main example, we need to discuss the effect of applying $p$-completion $L_p$ to a certain tower of spaces. Suppose $\{X_n\}_{n\in\mathbb{N}}$ is a tower of simply connected based spaces with finitely generated homotopy groups. Assume further that the layers are also simply connected and consider the corresponding tower $T = \{\Omega^2 X_n\}$ of two-fold loop spaces and maps. We wish to compare the tower $T$ with a tower $T(\ZZ_p)$
\[
\dots \to \Omega^2 L_p X_n \to \Omega^2 L_p X_{n-1} \to \dots \to \Omega^2 L_p X_0
\]
and an analogous spectral sequence whose $E^1$-page is 
\[
E^1_{-s,t}(T(\ZZ_p)) = 
\pi_{t-s} \Omega^2 \hofiber (L_p X_s \to L_p X_{s-1}) \; .
\]
A warning must be issued at this point: taking homotopy fibers does not commute with $L_p$ in general. However, under the simple connectivity assumption, we have that
\[
\hofiber (L_p X_s \to L_p X_{s-1}) \simeq L_p \hofiber (X_s \to X_{s-1}) \; .
\]
Therefore,
\begin{prop}\label{prop: complete E}
For each positive integer $r$, the map
\[
E^r_{-s,t}(T) \to E^r_{-s,t}(T(\ZZ_p))\;
\]
is given by algebraic $p$-completion, i.e., tensoring with $\ZZ_p$.
\end{prop}
\begin{proof}
For $r = 1$, this follows from what was just said and the fact that if a simply connected space $Y$ has homotopy groups that are finitely generated abelian, then the map 
\[
\pi_* Y \to \pi_* L_p Y
\] 
is a $p$-completion, in the sense that it is isomorphic to the canonical map
\[\pi_* Y \to \pi_* Y \otimes \ZZ_p \; .\]
For a finite $r$, the result follows by the exactness the functor $A\mapsto A\otimes\ZZ_p$.
\end{proof}

Now, suppose that we are in a situation for which there is convergence of the spectral sequence of the tower $T$, e.g. if the connectivity of the layers of $T$ grows to $\infty$. Then we have a tower approximating $\pi_*(\mathrm{holim}T(\ZZ_p))$
\begin{equation}\label{eq:towerFps}
\dots \to F_p^s \to F_p^{s-1} \to \dots \to F_p^{0}
\end{equation}
where $F_p^s$ is the image of $\pi_* (\mathrm{holim}T(\ZZ_p)) \to \pi_* T(\ZZ_p)_s$, 
and satisfying
\[
E^{\infty}_{-s, s+*}(T(\ZZ_p)) \cong \ker( F_p^s \to F_p^{s-1} ) \;.
\]
This is related to the tower $\{F^s\}$ which approximates $\pi_*(\mathrm{holim} T)$. Namely, the tower $(\ref{eq:towerFps})$ is a $p$-completion of the tower $(\ref{eq:towerFs})$.

\begin{rem}
Without convergence hypothesis, the tower $(\ref{eq:towerFps})$ may \emph{not} be a $p$-completion of the tower $(\ref{eq:towerFs})$. This is because 
 $p$-completion can fail to commute with infinite limits. For the same reason, the map
\[\pi_*(\mathrm{holim}T) \to \pi_* (\mathrm{holim}T(\ZZ_p))\]
may not be a $p$-completion. 
\end{rem}

\subsection{The Goodwillie-Weiss tower for long knots}

Let $d \geq 3$. The Goodwillie-Weiss tower approximating the space of long knots $\emb_c(\RR^1, \RR^d)$
\[
\dots \to T_k \emb_c(\RR^1, \RR^d) \to \dots \to T_2 \emb_c(\RR^1, \RR^d)  
\]
is a tower of $2$-fold loop spaces. We briefly explain this here, following \cite{BoavidaWeiss}. The $k$-th term, $k \geq 2$, can be described via a homotopy fiber sequence 
\[
 T_k \emb_c(\RR^1, \RR^d) \to \Omega \textup{InjLin}(\RR^1, \RR^d) \xrightarrow{\Omega f} \Omega \map^h(E_1, E_d)_{\leq k}
\]
where $\textup{InjLin}$ denotes the space of injective linear maps, and $\map^h(E_1, E_d)_{\leq k}$ denotes the space of derived maps between $k$-truncated operads. (That is, we restrict the operads to operations with at most $k$ inputs.)
The map $f$ is essentially an evaluation map (see \cite{BoavidaWeiss} for a precise description using configuration categories).

The space $\textup{InjLin}(\RR^1, \RR^d)$ is weakly equivalent to $\map^h(E_1, E_d)_{\leq 2}$ and with that in mind it is not hard to see that $f$ is a section of the truncation map
\[
g_k : \map^h(E_1, E_d)_{\leq k} \to \map^h(E_1, E_d)_{\leq 2} \; .
\]
Therefore, we have a weak equivalence
\[T_k \emb_c(\RR^1, \RR^d) \simeq \Omega^2 \hofiber(g_k).\]
Moreover, the map 
\[T_{k+1} \emb_c(\RR^1, \RR^d)\to T_k \emb_c(\RR^1, \RR^d)\]
can be identified with the double loop map of the obvious map 
\[\hofiber(g_{k+1})\to \hofiber(g_k)\;.\]
In order to simplify notation, we denote the tower $k \mapsto \Omega^2 \hofiber(g_k)$ by $T$.

\medskip
For a positive integer $n$, we write $\uli{n}$ for the finite set $\{1,\ldots,n\}$. The first page of the spectral sequence associated to $T$ has the following description: whenever $s > 2$, we have
\[
E^1_{-s,t}(T) = \pi_{t} \textup{thofiber} \left (S \subset \uli{s} \mapsto \emb(S, \RR^d) \right )
\]
where $\textup{thofiber}$ means the total homotopy fiber of the cube ; when $s\leq 2$, we have $E^1_{-s,t}(T)=0$. This identification follows from Göppl's thesis:

\begin{prop}\label{prop:flo}
Let $P$ be an operad having weakly contractible $P(0)$ and $P(1)$. Let us fix a map $E_1\to P$ that we use as a base point for each of the spaces $\map^h(E_1, P)_{\leq n}$. For $n \geq 2$, there is a homotopy fiber sequence
\[
\Omega^{n-2} \textup{thofiber} \left (S \subset \uli{n} \mapsto P(S) \right) \to \map^h(E_1, P)_{\leq n} \to \map^h(E_1, P)_{\leq n-1} 
\]
where the $n$-cube $S \mapsto P(S)$ is given by compositions with $0$-arity operations. 
\end{prop}
\begin{proof}
The main result of Göppl's thesis, the first theorem on p. 5 in {\cite[First Theorem, p. 5]{goppl}}, identifies the homotopy fiber of the restriction map 
\[
\map^h(E_1, P)_{\leq n} \to \map^h(E_1, P)_{\leq n-1}
\]
with the space of $\Sigma_n$-equivariant maps $E_1(n) \to P(n)$ extending the solid diagram of $\Sigma_n$-equivariant maps
\begin{equation*}
	\begin{tikzpicture}[descr/.style={fill=white}, baseline=(current bounding box.base)]
	\matrix(m)[matrix of math nodes, row sep=2.5em, column sep=2.5em,
	text height=1.5ex, text depth=0.5ex]
	{
	\partial E_1(n) & \partial P(n)	\\
	E_1(n) & P(n) \\
	\holimsub{S \subsetneq \uli{n}} E_1(S) & \holimsub{S \subsetneq \uli{n}} P(S) \\
	};
	\path[->,font=\scriptsize]
		(m-1-1) edge node [auto] {} (m-1-2)
		(m-1-1) edge node [auto] {} (m-2-1)
		(m-1-2) edge node [auto] {} (m-2-2)
		(m-2-1) edge node [auto] {} (m-3-1)
		(m-2-2) edge node [auto] {} (m-3-2)
		(m-3-1) edge node [auto] {} (m-3-2);
			\path[dashed,->,font=\scriptsize]
					(m-2-1) edge node [auto] {} (m-2-2);
	\end{tikzpicture}
\end{equation*}
where $\partial P(n)$ is the \emph{boundary} of $P(n)$, a certain homotopy colimit parametrized by trees  with exactly $n$ leaves and whose vertices have at most $n-1$ incoming edges (see \cite[2.1.12]{goppl} for a precise definiton). The horizontal maps are determined by the choice of basepoint. For the case of $E_1$, the inclusion $\partial E_1(n) \hookrightarrow E_1(n)$ is modelled by $\partial D^{n-2} \times \Sigma_{n} \hookrightarrow D^{n-2} \times \Sigma_n$ for $n \geq 2$, as in the inclusion of the boundary in (the symmetrization of) Stasheff's polyhedron.
\end{proof}

The homotopy groups of this total homotopy fiber are well-understood. We record this in the following two elementary and well-known statements.

\begin{lem}\label{lem:thofiber of split cube}
Let $X$ and $Y$ be two $(n-1)$ dimensional cubes. Let $p:X\to Y$ and $s:Y\to X$ be two maps of cubes with $s$ a homotopy section of $p$ (i.e. $p\circ s$ is homotopic to the identity). Then, for any $k\geq 1$, if we denote by $C$ the $n$-dimensional cube given by the map $p$, we have
\[\pi_k(\textup{thofiber}(C))\cong\ker\left( \pi_k(\textup{thofiber}(X))\to \pi_k(\textup{thofiber}(Y))\right)\]
where the map on the right hand side is the map induced by $p$.
\end{lem}

\begin{proof}
In general, there is a fiber sequence
\[\textup{thofiber}(C)\to \textup{thofiber}(X)\xrightarrow{p} \textup{thofiber}(Y)\]
and the existence of the map $s$ implies that the long exact sequence of homotopy groups splits into short exact sequences:
\[0\to\pi_k(\textup{thofiber}(C))\to \pi_k(\textup{thofiber}(X))\to\pi_k(\textup{thofiber}(Y))\to 0 \; \]
as claimed.
\end{proof}

\begin{prop}\label{prop : thof}
Under the hypothesis of proposition \ref{prop:flo} we have 
\[\pi_*\textup{thofiber} \left (S \subset \uli{n} \mapsto P(S) \right)\cong\bigcap_{i=0}^{n-1}\left(\ker\pi_*(s^i):\pi_*P(n)\to\pi_*P(n-1)\right)\]
where the maps $s^i:P(n)\to P(n-1)$ are the $n$ edges of the cube originating from $P(n)$.
\end{prop}

\begin{proof}
Let us call $C$ the cube under consideration. Let $z$ be any point of the space $P(0)$ (the choice does not matter since this space is contractible). The map $s^i$ is then defined as operadic composition with $z$ in the $(i+1)$-st input. The cube $C$ can be viewed as a map of $(n-1)$-cubes $X\to Y$ with $X(S)=P(S\cup \{n\})$ and $Y(S)=P(S)$ for $S\subset \uli{n-1}$. The map from $X$ to $Y$ is induced by operadic composition with $z$ in the input labelled $n$.

Recall that we have chosen a map $E_1\to P$. Let $x$ in $P(2)$ be the image of any point in $E_1(2)$ (again the choice is irrelevant). Using operadic composition with $x$ we obtain a homotopy section $Y\to X$ as in the previous lemma. By the previous lemma, we deduce that $\pi_*\textup{thofiber}(C)$ is the kernel of the map induced by $s^{n-1}$ on $\pi_*\textup{thofiber}(X)$. Iterating this reasoning $(n-1)$ more times, we obtain the desired result.
\end{proof}

We now fix a prime $p$. As in the discussion in the beginning of the section, we want to investigate the tower
\[
\dots \to \Omega^2 L_p \hofiber(g_k) \to \Omega^2 L_p \hofiber(g_{k-1}) \to \dots
\]
In line with the notation of the previous subsection, we denote this tower by $T(\ZZ_p)$. We also introduce another tower -- denoted $T'(\ZZ_p)$ -- of the form
\[
\dots \to \Omega^2 \hofiber(g'_k) \to \Omega^2 \hofiber(g'_{k-1}) \to \dots
\]
where $g'_k$ is the truncation map for the operad $L_pE_d$ :
\[
g'_k : \map^h(E_1, L_pE_d)_{\leq k} \to \map^h(E_1, L_pE_d)_{\leq 2} \; .
\]

\begin{lem}\label{lem: equiv towers}
The canonical map $T(\ZZ_p)\to T'(\ZZ_p)$ is a weak equivalence of towers. As such, it induces an isomorphism of the associated spectral sequences.
\end{lem}
\begin{proof}
We compare the layers. According to Propositions \ref{prop:flo} and \ref{prop : thof}, the $n^{th}$ layer of $T^\prime(\ZZ_p)$ has homotopy groups
 \[
\bigcap_{i=0}^{n-1}\left(\ker\pi_*(s^i): \pi_*E_d(n) \otimes\ZZ_p \to \pi_*E_d(n-1)\otimes\ZZ_p\right) \; .
 \]

By Proposition \ref{prop: complete E}, the homotopy groups of the layers of $T(\ZZ_p)$ are the $p$-completion of the homotopy groups of the layers of $T$. The homotopy groups of the layers of $T$ are also identified using Propositions \ref{prop:flo} and \ref{prop : thof}. Taking $p$-completion amounts to tensoring the homotopy groups with $\ZZ_p$, which is an exact functor since $\ZZ_p$ is torsion-free, and so the result follows.
\end{proof}

The important corollary of the above discussion is the following.

\begin{cor}\label{cor:action on tower}
The tower $T(\ZZ_p)$ has an action of $GT_p$. Therefore, the spectral sequence $E^r_{-s,t}(T(\ZZ_p))$ has an action of $GT_p$.
\end{cor}

\begin{proof}
By Theorem \ref{thm:GtonLpE} we deduce that $GT_p$ acts on the tower $T^\prime(\ZZ_p)$ fixing basepoints (which are determined by a fixed choice of inclusion $E_1 \to E_d$ and the canonical map $E_d \to L_p E_d$) and so, by Lemma \ref{lem: equiv towers}, it acts on the tower $T(\ZZ_p)$ fixing basepoints.
\end{proof}

\begin{rem}
In \cite[Appendix D]{KTchord}, Kassel and Turaev construct a Galois action on the Vassiliev tower over $\QQ_p$. As we show later, the Vassiliev tower over $\QQ_p$ is isomorphic to the Goodwillie-Weiss tower $\pi_0 T(\QQ_p)$ (defined in the next section). But we do not know whether the action from Corollary \ref{cor:action on tower} agrees with the one of Kassel and Turaev.

It is also not absurd to wonder whether there is a relation between the tower $T'(\ZZ_p)$ and the \emph{profinite knots} of Furusho \cite{Furusho}, who also come with an action of the Grothendieck-Teichm\"uller group. 
\end{rem}

\begin{rem}
There is a closely related tower
\[
\dots \to \overline{T}_k \to \dots \to \overline{T}_2
\]
approximating the space
\[
\overline{\emb}_c(\RR^1, \RR^n) = \hofiber ({\emb}_c(\RR^1, \RR^n) \to \Omega \textup{InjLin}(\RR^1, \RR^d) )
\]
and whose $k$-term is
\[
\overline{T}_k =  \hofiber(T_k \to \Omega \textup{InjLin}(\RR^1, \RR^d)) \simeq \Omega^2 \map^h(E_1, E_d)_{\leq k} \; .
\]
This tower is identified (Sinha \cite{sinhaoperads}) with the tower associated to a pointed cosimplicial space $K^\bullet$
\[
\dots \to \holim_{\Delta_{\leq k}} K^\bullet \to \holim_{\Delta_{\leq k-1}} K^\bullet \to \dots \to \holim_{\Delta_{\leq 0}} K^\bullet = K^0
\]
The first page of the spectral sequence associated to this tower $\overline{T}$ agrees with that of the tower $T$, except on the column $s = 2$. All the discussion above applies to the tower $\overline{T}$ without any additional difficulty. In section \ref{sec:homology} we study the Bousfield-Kan homology spectral sequence of this cosimplicial space.
\end{rem}

\section{Main theorem}
In this section, we prove the following theorem from which Theorem B of the introduction is an easy consequence.

\begin{thm}\label{thm:htpyBKGW}
Let $p$ be a prime number. The differential
\[
d^r_{-s,t} : E^{r}_{-s,t}(T(\ZZ_p)) \to E^{r}_{-s-r, t+r-1}(T(\ZZ_p))
\]
vanishes if $r-1$ is not a multiple of $(p-1)(d-2)$ and if
\[
t < 2p - 2 + (s-1)(d-2) \; .
\]
\end{thm}

We begin by recalling the theorem of Hilton-Milnor in a form that we will need. Given a finite set $R$, we denote by $W(R)$ a set of Lie words in the finite set $R$ that forms a basis for the free Lie algebra on $R$. For $w\in W(R)$ we write $|w|$ for its length.

\begin{thm}[Hilton-Milnor]\label{thm:Hilton Milnor}
There is a weak equivalence
\[
{\prod_{w\in W(R)}}^{\!\!\!\!\!\!\!\prime} \Omega S^{|w|(d-2) + 1}\xrightarrow{\simeq} \Omega\left( \bigvee_RS^{d-1} \right) 
\]
where the symbol $\prod'$ stands for the weak product (union of finite products).
\end{thm}

\medskip

Let  $T(\QQ_p)$ be the tower analogous to $T(\ZZ_p)$ obtained by replacing all instances of $L_p Z$, $Z$ a space, by their rationalizations $(L_p Z)_\QQ$. As in the previous section, we have a $GT_p$-action on $T(\QQ_p)$ and so a $GT_p$-action on the associated spectral sequence $E(T(\QQ_p))$ (c.f. Corollary \ref{cor:action on tower}). Note also that by exactness of the functor $-\otimes_{\ZZ_p}\QQ_p$, we have the following isomorphisms for any $r$, $s$ and $t$
\[E^r_{-s,t}(T(\QQ_p))\cong E^r_{-s,t}(T(\ZZ_p))\otimes\QQ\cong E^r_{-s,t}(T)\otimes\QQ_p\]
the first of which is $GT_p$-equivariant.

\begin{thm}\label{thm:htpy-cycl}
Let $d \geq 3$. The $\QQ_p$-vector space 
\[
E^{1}_{-s,t}(T(\QQ_p))
\]
is zero unless $t = n(d-2) + 1$ for some $n  \geq s-1$. If $t = n(d-2) + 1$, the $GT_p$-action is cyclotomic of weight $n$.
\end{thm}
\begin{proof}
We assume $s > 2$, otherwise $E^1_{-s,t}(T(\QQ_p))$ is zero. Then
\[
E^{1}_{-s,t}(T(\QQ_p)) = \pi_{t} \textup{thofiber} \left (S \subset \uli{s} \mapsto L_p \emb(S, \RR^d) \right ) \otimes \QQ \; .
\]
Regard the $s$-cube $S \mapsto L_p \emb(S, \RR^d)$ as an $(s-1)$-cube of maps
\[
( R \subset \uli{s-1}) \mapsto \chi_R 
\]
where $\chi_R$ is the map $L_p \emb(R \cup \{s\}, \RR^d) \to L_p \emb(R, \RR^d)$ forgetting the point labelled \emph{s}. Then the total homotopy fiber of the original cube is identified with the total homotopy fiber of the $(s-1)$-cube $R \mapsto \hofiber \chi_R$. This $(s-1)$-cube is identified with a cube
\[
R \mapsto \phi_R:=L_p( \vee_{R} S^{d-1} )\;.
\]
where the maps collapse wedge summands. This follows from the Fadell-Neuwirth fiber sequence 
\[
\vee_{R} S^{d-1} \to \emb(R \cup \{s\}, \RR^d) \to \emb(R, \RR^d) \; 
\]
and the fact that $L_p$ commutes with taking homotopy pullbacks of simply connected spaces.
By Proposition \ref{prop : thof}, the canonical projection
\begin{equation}\label{eq:projmap}
\textup{thofiber}(\phi) \to L_p \left ( \vee_{s-1} S^{d-1} \right )
\end{equation}
is injective on homotopy groups. Now, the $GT_p$ action on $E^{1}_{-s,t}(\QQ_p)$ comes from a (basepoint preserving) action on the cube $S \mapsto L_p \emb(S, \RR^d)$, or equivalently, from an action on the cube $\phi$. As such, the projection map (\ref{eq:projmap}) is $GT_p$-equivariant. This means that to understand the action on $E^{1}_{-s,t}(\QQ_p)$ it is enough to understand the action on 
\[
\pi_{t} L_p \left ( \vee_{s-1} S^{d-1} \right ) \otimes \QQ \cong \pi_{t} \left ( \vee_{s-1} S^{d-1} \right ) \otimes \QQ_p \; .
\]

Using Hilton-Milnor theorem (Theorem \ref{thm:Hilton Milnor}) or, alternatively, Milnor-Moore, the rational homotopy groups of a wedge of $(d-1)$-dim spheres are concentrated in degrees $n(d-2)+1$, for $n \geq 1$. The integer $n$ records the \emph{size} of the word in the Hilton-Milnor decomposition. In the total homotopy fiber of the $(s-1)$-cube $R \mapsto \vee_R S^{d-1}$, $n$ must be at least $s-1$ since the generator $\iota \in \pi_{d-1}(S^{d-1})$ of each wedge summand must occur at least once.
 This establishes the first part of the theorem. The second part is a consequence of the fact that the $GT_p$-action on $\pi_{n(d-2)+1}(\vee S^{d-1}) \otimes \QQ_p$ is cyclotomic of weight $n$. This is the content of Proposition \ref{prop:GTonwedges} below when $n = 1$. For higher $n$'s, these homotopy groups are generated by Whitehead products of elements in $\pi_{d-1}$, and since the $GT_p$ action exists at the level of spaces it must be compatible with Whitehead products.
\end{proof}

\begin{prop}\label{prop:GTonwedges}
The $GT_p$ action on
\[
\pi_{d-1} L_p(\vee_k S^{d-1}) \cong \oplus_{k} \ZZ_p
\] 
is cyclotomic of weight $1$.
\end{prop}
\begin{proof}
For the duration of this proof, we will drop $L_p$ from the notation and so, for a space $X$, we keep $X$ as notation for $L_p X$.

By Proposition \ref{prop: weight}, the statement holds for $k = 1$. For $k > 1$, set $\ell = k+1$ and consider the commutative square
\begin{equation*}
	\begin{tikzpicture}[descr/.style={fill=white}, baseline=(current bounding box.base)]
	\matrix(m)[matrix of math nodes, row sep=2.5em, column sep=5em,
	text height=1.5ex, text depth=0.5ex]
	{
	\emb(\uli{\ell}, \RR^d) &  \emb(\uli{k}, \RR^d) 	\\
	\emb(\uli{2}, \RR^d) & \emb(\uli{1}, \RR^d) \\
	};
	\path[->,font=\scriptsize]
		(m-1-1) edge node [auto] {\mbox{forget $k+1$}} (m-1-2)
		(m-1-1) edge node [auto] {} (m-2-1)
		(m-2-1) edge node [auto] {\mbox{forget $2$}} (m-2-2)
		(m-1-2) edge node [auto] {} (m-2-2);
	\end{tikzpicture}
\end{equation*}
where the left map is induced by the inclusion $f : \uli 2 \to \uli \ell$ given by $f(1) = j$, $f(2) = k+1$; and the right map is induced by the map $\uli 1 \to \uli k$ selecting $j$. The induced map on horizontal homotopy fibers has the form
\[
p_j : \vee_k S^{d-1} \to S^{d-1}
\]
and it corresponds to collapsing all summands except the $j$-th one to the basepoint. All the (based) maps in the square are $GT_p$-equivariant since they correspond to composition maps in the operad $E_d$ involving nullary operations, and so $p_j$ is a $GT_p$-equivariant map. 

Let $\alpha$ be an element of $GT_p$. To describe the action of $\alpha$ on the wedge $\vee_k S^{d-1}$ we must describe, for each pair $1 \leq i, j \leq k$, the composition
\[
S^{d-1} \xrightarrow{\textup{incl}_i} \vee_k S^{d-1} \xrightarrow{\alpha} \vee_k S^{d-1} \xrightarrow{p_j} S^{d-1}
\] 
where $\textup{incl}_i$ means the inclusion of the $i$-th summand. Since $p_j   \circ  \alpha = \alpha \circ p_j$, this composite equals $\alpha  \circ  p_j  \circ  \textup{incl}_i$, which is $\alpha$ if $i = j$ and is trivial otherwise.
\end{proof}

\begin{proof}[Proof of Theorem \ref{thm:htpyBKGW}]
Let $W$ denote the total homotopy fiber of the $(s-1)$-cube $R \mapsto \vee_R S^{d-1}$, for $s > 2$. We look at the $p$-torsion in the homotopy groups of $W$. By the naturality in the Hilton-Milnor theorem (Theorem \ref{thm:Hilton Milnor}), it follows that
\[
\Omega W \simeq {\prod_{w}}^{\prime} \Omega S^{|w|(d-2)+1}
\]
where $w$ runs over the words in the letters $1, \dots, s-1$ containing every letter \emph{at least once}. As such, the smallest word in the product has length $s-1$, and the sphere of smallest dimension is $S^{(s-1)(d-2)+1}$. 

By a famous result of Serre, the homotopy groups $\pi_*S^{\ell}$ are $p$-torsion free for $* \leq \ell+2p-4$ for $\ell \geq 3$.
Therefore, the homotopy groups $\pi_* W$ are $p$-torsion free whenever $*  \leq N$ where
\[
N = (s-1)(d-2)+2p-3 \; ,
\]
(remember $d \geq 3$ and $s \geq 3$ so that $(s-1)(d-2)+1 \geq 3$). And so, in the range $t \leq N$, we have an inclusion
\[
E^1_{-s,t}(T(\ZZ_p)) \hookrightarrow E^1_{-s,t}(T(\ZZ_p)) \otimes \QQ\cong E^1_{-s,t}(T(\QQ_p)) \; .
\]
Therefore, by Theorem \ref{thm:htpy-cycl}, in the range $t \leq N$, the group $E^1_{-s,t}(T(\ZZ_p))$ is zero unless $t = n(d-2) +1$, in which case the $GT_p$ action is cyclotomic of weight $n$. The same is true for the successive pages $E^r_{-s,t}$ by Proposition \ref{prop: submodule and quotient} and Remark \ref{rem : Zp linear}. Proposition \ref{prop: no map} completes the proof.
\end{proof}

We can now prove Theorem B of the introduction.

\begin{proof}[Proof of Theorem B]
Recall that $\ZZ_{(p)}$ denotes the ring of $p$-local integers. Proposition \ref{prop: complete E} gives an isomorphism of spectral sequences
\[E^*_{*,*}(T(\ZZ_p))\cong E^*_{*,*}(T)\otimes_{\ZZ}\ZZ_{p}\cong (E^*_{*,*}(T)\otimes_{\ZZ}\ZZ_{(p)})\otimes_{\ZZ_{(p)}}\ZZ_p\]
Then the result follows from the fact that a map $f$ of $\ZZ_{(p)}$-modules vanishes if and only if $f\otimes_{\ZZ_{(p)}}\ZZ_p$ vanishes.
\end{proof}

\section{Some consequences}

\begin{cor}\label{cor: collapse over Q}
The spectral sequence $E_{-s,t}^*(T)\otimes_{\ZZ}\QQ$ collapses at the $E^2$-page.
\end{cor}

\begin{proof}
Let $d^r_{-s,t}$ be a differential in that spectral sequence with $r>2$. Up to choosing $p$ a large enough prime, we may assume that $t<2p-2+(s-1)(d-1)$ and that $r-1$ is not a multiple of $(p-1)(d-1)$. Then this differential is zero in $E^r_{-s,t}(T)\otimes_{\ZZ}\ZZ_{(p)}$ by Theorem B so it is also zero after inverting $p$.
\end{proof}

\begin{rem}
For $d\geq 4$, this result is due to Arone, Lambrechts, Turchin and Voli\'{c} (see \cite{ALTV}). The case $d=3$ does not appear in \emph{loc. cit.} and, to the best of our knowledge, does not appear elsewhere in the literature. The reason seems to be that the relative formality of the map of operads $E_1\to E_3$ was not known when \cite{ALTV} was written. 
\end{rem}

\begin{cor}\label{cor : collapse for T_n} 
Let $p$ be a prime. For $n \leq (p-1)(d-2) + 3$, the spectral sequence associated to the tower $T_{\leq n}(\ZZ_p)$, computing 
\[
\pi_* T_{n} \emb_c(\RR^1, \RR^d) \otimes \ZZ_p \; ,
\]
collapses at the $E^2$-page for $(-s,t)$ satisfying $t< 2p-2+(s-1)(d-2)$.
\end{cor}

\begin{proof}
Let $T(s) := 2p-2+(s-1)(d-2)$. Let $A$ be the region of the second-quadrant consisting of those $(-s,t)$ such that $t < T(s)$. By Theorem \ref{thm:htpyBKGW}, the first possibly non-trivial differential in the region $A$ is
\[
d^{R} : E^{R}_{-s,t} \to E^{R}_{-s-R,t+R-1} \; .
\]
with $R = (p-1)(d-2)+1$. Since the first non-zero column is $s = 3$, the first possibly non-trivial differentials land in $E^{R}_{-3-R,t+R-1}$. But the groups $E^1_{-s,*}$ are zero whenever $s > n$. So $E^{R}_{-3-R,t+R-1}$ is zero, and so are the target groups of higher differentials from the region $A$. Therefore, the spectral sequence collapses in the region $A$ whenever $n < R + 3$.
\end{proof}

\begin{rem}
Note that the corollary has nothing to offer at the prime $2$.
\end{rem}

\begin{rem}
Theorem \ref{thm:htpyBKGW} has something to say about the homotopy groups of $\pi_*(T_n)$ even outside of the range of the previous corollary. For instance, let us consider $\pi_0(T_n)\otimes\ZZ_p$ in the case of knots in $\RR^3$. This abelian group has filtration whose associated graded is $\oplus_{s\leq n}E^{\infty}_{-s,s}(T_n(\ZZ_p))$. The corollary tells us that we have isomorphisms
\[E^2_{-s,s}\cong E^\infty_{-s,s}\]
in the range $s\leq p+2$. For $p+3 \leq s\leq 2p+1$, the only differential that can hit $E^2_{-s,s}$ is $d^p$, therefore, we have an isomorphism
\[E^\infty_{-s,s}\cong E^{p+1}_{-s,s}\cong E^2_{-s,s}/\mathrm{Im}(d^p)\]
For $2p+2\leq s\leq 3p$, the group $E^{\infty}_{-s,s}$ will be the quotient of $E^2_{-s,s}$ by the image of $d^p$ and $d^{2p-1}$. This pattern continues.
\end{rem}

In the remainder of this section, we prove some results about extensions leading to the proof of Theorem \ref{thmC} below and Corollary C from the introduction. Let us fix a unit $u\in\mathbb{Z}_p$ such that the residue of $u$ modulo $p$ is a generator of $\mathbb{F}_p^\times$. For $M$ a $\ZZ_p$-module and $i$ an integer, we denote by $M(i)$ the $\ZZ_p[t,t^{-1}]$-module $M$ where $t$ acts as multiplication by $u^i$. 

\begin{prop}
Let $M$ be a $\ZZ_p$-module of finite type. Let $i$ and $j$ be two integers such that $(p-1)$ does not divide $i-j$. Then the group $\mathrm{Ext}^1_{\mathbb{Z}_p[t,t^{-1}]}(M(i),N(j))$ is zero.
\end{prop}

\begin{proof}
Since $M$ can be written as finite direct sums of copies of $\ZZ_p$ and $\ZZ/p^n$, wthout loss of generality, we may assume that $M$ is either $\ZZ_p$ or $\ZZ/p^n$. The short exact sequence
\[0\to \ZZ/p^n\to\ZZ/p^{n-1}\to\ZZ/p\to 0\]
induces an exact sequence
\[\mathrm{Ext}^1_{\mathbb{Z}_p[t,t^{-1}]}(\ZZ/p(i),N(j))\to\mathrm{Ext}_{\mathbb{Z}_p[t,t^{-1}]}^1(\ZZ/p^{n+1}(i),N(j))\]
\[\to\mathrm{Ext}^1_{\mathbb{Z}_p[t,t^{-1}]}(\ZZ/p^n(i),N(j))\]
Therefore, by induction on $n$, we can reduce the case $M=\ZZ/p^n$ to the case $M=\ZZ/p$. The short exact sequence
\[0\to\ZZ_p\xrightarrow{.p}\ZZ_p\to\ZZ/p\to 0\]
induces an exact sequence
\[\mathrm{Hom}_{\ZZ_p[t,t^{-1}]}(\ZZ_p(i),N(j))\to\mathrm{Ext}^1_{\mathbb{Z}_p[t,t^{-1}]}(\ZZ/p(i),N(j))\to\mathrm{Ext}^1_{\mathbb{Z}_p[t,t^{-1}]}(\ZZ_p(i),N(j))\]
therefore, the case $M=\ZZ/p$ follows from the vanishing of $\Hom_{\ZZ_p[t,t^{-1}]}(\ZZ_p(i),N(j))$ and of $\mathrm{Ext}^1_{\mathbb{Z}_p[t,t^{-1}]}(\ZZ_p(i),N(j))$. The vanishing of the $\mathrm{Hom}$-group is similar to the proof of Proposition \ref{prop: no map}. In order to compute the $\mathrm{Ext}$-group, we use the following projective resolution of $\ZZ_p(i)$ 
\[0\to \ZZ_p[t,t^{-1}]\xrightarrow{.(t-u^i)}\ZZ_p[t,t^{-1}]\to\ZZ_p\to 0.\]
We obtain an exact sequence
\[N\to N\to \mathrm{Ext}^1_{\ZZ_p[t,t^{-1}]}(\ZZ_p(i),N(j))\to 0\]
where the first map is multiplication  by $u^j-u^i$. Since $j-i$ is not a multiple of $(p-1)$, the $p$-adic number $u^j-u^i$ is a unit and the result follows.
\end{proof}

\begin{prop}\label{prop : trivial extensions}
Let $M$ be a $\ZZ_p[GT_p]$-module of finite type as a $\ZZ_p$-module. Let $n$ and $m$ be two integers with $0\leq n-m< p-1$. Assume that $M$ sits at the top of a tower of $\ZZ_p[GT_p]$-modules
\[M=M_n\xrightarrow{p_n} M_{n-1}\xrightarrow{p_{n-1}} \ldots\to M_m\xrightarrow{p_m} 0\]
where each map $p_k$ is surjective. Assume further that, for each $k$, the $GT_p$-action on the kernel $F_k$ of $p_k$ is cyclotomic of weight $k$. Then there is a non-canonical isomorphism of $\ZZ_p$-modules
\[M\cong F_n\oplus\ldots\oplus F_m \; .\]
\end{prop}

\begin{proof}
It suffices to prove that, for each $k$, the short exact sequence
\[0\to F_k\to M_k\to M_{k-1}\to 0\]
is split. The obstruction lies in the $\mathrm{Ext}$-group $\mathrm{Ext}^1_{\ZZ_p}(M_{k-1},F_k)$. Since this short exact sequence is a sequence of  $\ZZ_p[GT_p]$-modules, the obstruction lies in the image of the obvious map
\[\mathrm{Ext}^1_{\ZZ_p[GT_p]}(M_{k-1},F_k)\to \mathrm{Ext}^1_{\ZZ_p}(M_{k-1},F_k)\;.\]
Pick an element $u$ of $\ZZ_p^{\times}$ such that the residue of $u$ modulo $p$ is a generator of $\FF_p^\times$. Pick a lift $t$ of $u$ in $GT_p$. This choice induces a homomorphism of $\ZZ_p$-algebras
\[\ZZ_p[t,t^{-1}]\to \ZZ_p[GT_p]\]
If we restrict along this map, the $\ZZ_p[GT_p]$-module $F_k$ is simply the module $F_k(k)$. Hence the obstruction lies in the image of the obvious map
\[\mathrm{Ext}^1_{\ZZ_p[t,t^{-1}]}(M_{k-1},F_k(k))\to \mathrm{Ext}^1_{\ZZ_p}(M_{k-1},F_k)\]
and it is enough to show that the group $\mathrm{Ext}^1_{\ZZ_p[t,t^{-1}]}(M_{k-1},F_k(k))$ is zero. We will in fact prove that the group $\mathrm{Ext}^1_{\ZZ_p[t,t^{-1}]}(M_{i},F_k)=0$ for any $m\leq i<k$. We prove this by induction on $i$. For $i=m$, we have $M_m=F_m(m)$ and so
\[\mathrm{Ext}^1_{\ZZ_p[t,t^{-1}]}(M_m,F_{k}(k))=0\]
by the previous proposition. If the statement is true for $i$, then, we use the long exact sequence induced by the short exact sequence of $\ZZ_p[t,t^{-1}]$-modules
\[0\to F_{i+1}(i+1)\to M_{i+1}\to M_i\to 0\]
and the previous proposition to deduce the statement for $i+1$.
\end{proof}

Using this proposition and the theorem below, Corollary C from the introduction will follow. We write $T_{\leq n}$ for the Goodwillie-Weiss tower truncated in degree $n$.

\begin{thm}\label{thmC}
Let $n$ be a positive integer. For any prime number $p$ satisfying $n \leq (p-1)(d-2) + 3$, there is a (non-canonical) isomorphism
\[
\pi_i T_{n} \emb_c(\RR^1, \RR^d) \otimes \ZZ_{(p)} \cong \oplus_{t-s = i} E^2_{-s,t}(T_{\leq n})\otimes\ZZ_{(p)} \; 
\]
for $i \leq 2(p + d) - 5$.
\end{thm}

\begin{proof}
The analogous statement with $\ZZ_{(p)}$ replaced by $\ZZ_p$ follows from Corollary \ref{cor : collapse for T_n} and the previous proposition. If $M$ and $N$ are finitely generated $\ZZ_{(p)}$-modules, the canonical map
\[\mathrm{Ext}_{\ZZ_{(p)}}(M,N)\otimes\ZZ_p\to\mathrm{Ext}_{\ZZ_p}(M\otimes\ZZ_p,N\otimes\ZZ_p)\]
is an isomorphism. This can be reduced to the case where $M$ and $N$ are either $\ZZ_{(p)}$ or $\ZZ/p^k$ in which case this is an easy computation. It follows that if an extension of $\ZZ_{(p)}$-modules splits over $\ZZ_p$, it also splits over $\ZZ_{(p)}$.
\end{proof}

\begin{proof}[Proof of Corollary C]
Goodwillie-Klein excision estimates imply that the map 
\[
\emb_c(\RR^1, \RR^d) \to T_n \emb_c(\RR^1, \RR^d)
\] is $n(d-3)$-connected (see \cite[Theorem A]{goodwilliemultiple}). Setting $n = (p-1)(d-2) + 3$ and combining this with Theorem \ref{thmC} we have that
\[
\pi_i \emb_c(\RR^1, \RR^d) \otimes \ZZ_{(p)} \cong \oplus_{t-s = i} E^2_{-s,t}(T_{\leq n}) \otimes\ZZ_{(p)}
\]
for

\[
i < \min(2(p + d) - 4, n(d-3)).
\]
This $\min$ is equal to $n=2p+1$ if $d=4$ and to $2(p + d) - 4$ if $d>4$.

Finally, since the restriction maps $T_{s} \emb_c(\RR^1, \RR^d) \to T_{s-1} \emb_c(\RR^1, \RR^d)$, for $s > n$, are $n(d-3)$-connected, it follows that the canonical map
\[
 E^1_{-s,t}(T) \to E^1_{-s,t}(T_{\leq n}) 
\]
is an isomorphism for $t-s < n(d-3)$. Hence, $E^2_{-s,t}(T) \cong E^2_{-s,t}(T_{\leq n})$ for $t - s < n(d-3) - 1$.
\end{proof}

\begin{rem}
Observe that, in the case $d=4$, the proof above actually gives a computation of one more homotopy group of $\emb_c(\RR^1, \RR^4)$. Namely, we have 
\[\pi_{2p}\emb_c(\RR^1, \RR^4)\otimes\ZZ_{(p)}\cong\bigoplus_{t-s=2p}E^2_{-s,t}(T_{\leq 2p+1})\otimes\ZZ_{(p)}.\]
\end{rem}

\section{Universality}

We now come to the proof of theorem A from the introduction. To lighten notations, we denote the space of knots $\emb_c(\RR^1, \RR^3)$ by $K$ in this section, as we did in the introduction. Given two knots $f, g$, we write $f \sim_{n-1} g$ if $f$ and $g$ share the same type $n-1$ invariants. This defines an equivalence relation on $\pi_0 K$; the set of equivalences classes is denoted $\pi_0(K)/\!\!\sim_{n-1}$. It is not hard to see that the operation of concatenation makes this set into a commutative monoid. This commutative monoid is in fact a finitely generated abelian group as shown in \cite{Gusarov}.

\begin{defn}
Let $R$ be a commutative ring. An \emph{additive Vassiliev invariant} of type $(n-1)$  over $R$ is a map
\[I:\pi_0K\to M\]
with $M$ an $R$-module such that 
\begin{itemize}
\item $I$ is a morphism of monoids,
\item If $f \sim_{n-1} g$, then $I(f)=I(g)$.
\end{itemize}
An additive Vassiliev invariant of type $(n-1)$ is called \emph{universal} if any other additive Vassiliev invariant of type $(n-1)$ has a unique factorization through it.
\end{defn}

Any two universal additive Vassiliev invariants of type $(n-1)$ are uniquely isomorphic. Tautologically, the quotient map
\[\pi_0(K)\to \pi_0(K)/\!\!\sim_{n-1}\]
is a universal additive Vassiliev invariant of type $(n-1)$ over $\mathbb{Z}$ and the composite
\[\pi_0(K)\to \pi_0(K)/\!\!\sim_{n-1}\to(\pi_0(K)/\!\!\sim_{n-1})\otimes R\]
is a universal additive Vassiliev invariant of type $(n-1)$ over $R$.

\begin{thm}[Kosanović]\label{thm:collapse implies universal}
Let $R$ be a commutative ring that is torsion-free (e.g. $\ZZ$, $\QQ$, $\ZZ_p$, $\ZZ_{(p)}$). The evaluation map
\[e_n : \pi_0 K \to \pi_0 T_n K\otimes R\]
is a universal Vassiliev invariant of degree $n-1$ over $R$ if the canonical map
\[
E^2_{-k,k} \otimes R \to E^\infty_{-k,k}  \otimes R
\]
is an isomorphism for all $k \leq n$.
\end{thm}

\begin{proof}
In \cite{BCKS}, and \cite{KST}, it it was shown that $e_n$ is well-defined on equivalence classes and as such factors through
$\overline{e}_n : (\pi_0(K)/\!\!\sim_{n-1}) \to \pi_0 T_n K$. The statement of the theorem is equivalent to the statement that $\overline{e}_n \otimes R$ is an isomorphism. 

We argue inductively, using the commutative square of group homomorphisms
\begin{equation*}
	\begin{tikzpicture}[descr/.style={fill=white}, baseline=(current bounding box.base)]
	\matrix(m)[matrix of math nodes, row sep=2.5em, column sep=5em,
	text height=1.5ex, text depth=0.5ex]
	{
	{} \pi_0(K)/\!\!\sim_{n-1} & \pi_0 T_n K	\\
	{}\pi_0(K)/\!\!\sim_{n-2} & \pi_0 T_{n-1} K	\\
	};
	\path[->,font=\scriptsize]
		(m-1-1) edge node [auto] {$\overline{e}_n$} (m-1-2)
		(m-1-1) edge node [auto] {} (m-2-1)
		(m-2-1) edge node [auto] {$\overline{e}_{n-1}$} (m-2-2)
		(m-1-2) edge node [auto] {} (m-2-2);
	\end{tikzpicture}
\end{equation*}
and compare the map between vertical kernels
\[
\Phi_{n-1} \to E^{\infty}_{-n,n}
\]
where $\Phi_{n-1}$ denotes the kernel of the left vertical map. In \cite{CTgrope}, Conant-Teichner construct a surjective homomorphism $R_n : E^2_{-n,n} \to \Phi_{n-1}$ and in \cite{Kosanovic}, Kosanović shows that the composition
\[
E^2_{-n,n} \to \Phi_{n-1} \to E^\infty_{-n,n} \;
\]
agrees with the canonical map. Since the composition is an isomorphism by hypothesis, we then have that both maps are isomorphisms. It follows by an application of the five lemma that $\overline{e}_n \otimes R$ is an isomorphism if the homomorphism $E^2_{-k,k} \otimes R \to E^\infty_{-k,k} \otimes R$ is an isomorphism for all $k\leq n$, as claimed.
\end{proof}

\begin{proof}[Proof of Theorem A]
The first part follows from Theorem \ref{thm:collapse implies universal} in conjunction with Theorem B. The second part is subsumed by Theorem \ref{thmC}.
\end{proof}

\section{The homology Goodwillie-Weiss spectral sequence}\label{sec:homology}
The Goodwillie-Weiss tower for $\overline{\emb}_c(\RR,\RR^d)$, the homotopy fiber of the inclusion
\[\emb_c(\RR,\RR^d)\to \mathrm{imm}_c(\RR,\RR^d)\;,\]
can also be described as follows. Consider the category $\sO_{k}$ whose objects are open proper subsets $T$ of $\RR$ containing $(-\infty,0]\cup[1,+\infty)$ and having at most $k+2$ connected components. A morphism $T \to R$ is an isotopy connecting $T \subset \RR$ to a subset of $R \subset \RR$, and which fixes $(-\infty,0]\cup[1,+\infty)$ pointwise. Then 
\[
T_k \overline{\emb}_c(\RR^1, \RR^d) = \holimsub{T \in \sO_{k}^\op} \overline{\emb}_c(T, \RR^d) \; .
\]
The Goodwillie-Weiss tower agrees with the tower associated to a cosimplicial space, the one coming from the usual filtration of $\Delta$ by the subcategories $\Delta_{\leq k}$ spanned by the objects $[n]$ satisfying $n \leq k$. This was first proved by Sinha \cite{sinhaoperads}. Roughly, the relation comes from the equivalence between the topological category $\mathcal{O}_{k}$ and the opposite of $\Delta_{\leq k}$. It follows that the functor $\overline{\emb}_c(-,\RR^d)$ defines a cosimplicial space whose homotopy limit is the limit of the Goodwillie-Weiss tower. In this section, we review an operadic construction of this cosimplicial space and use it to study the associated homology Bousfield-Kan spectral sequence.

\medskip
Recall that a multiplicative operad is a non-symmetric topological or simplicial operad $P$ together with the data of a map $\alpha:A\to P$ where $A$ denotes the non-symmetric associative operad. From such data, one can construct a cosimplicial space $X^\bullet$ as we now recall (see \cite[Definition 2.17]{sinhaoperads}). We denote by $u$ the image in $P(0)$ of the unique point in $A(0)$ and by $m$ the image in $P(2)$ of the unique point in $A(2)$.

In degree $q$, our cosimplicial space will be given by $X^q=P(q)$. The cofaces $d^i:P(q)\to P(q+1)$ are as follows. The two outer cofaces $d^0$ and $d^{q+1}$ take $x\in P(q)$ to $m\circ_1 x$ and $m\circ_2x$ respectively. The inner coface $d^i$ with $i\in\{1,\ldots,q\}$ takes $x$ to $x\circ_i m$. The codegeneracy $s^i: P(q)\to P(q-1)$ takes $x\in P(q)$ to  $x\circ_iu$.

In what follows, we will start with a map $\tilde{A}\to P$ where $\tilde{A}$ is not quite the associative operad but is merely weakly equivalent to it. The following proposition will be useful in that situation. We denote by $\cat{Op}_{ns}$ the category of non-symmetric operads in simplicial sets.

\begin{prop}\label{prop: rigidification}
Let $\tilde{A}$ be a non-symmetric operad and let $w:\tilde{A}\to A$ be a map of non-symmetric operads. Consider the adjunction
\[w_!:\cat{Op}_{ns}^{\tilde{A}/}\leftrightarrows \cat{Op}_{ns}^{A/}:w^*\]
where $w^*$ is precomposition by $w$ and $w_!$ is its left adjoint. Then, if $w$ is a weak equivalence, the adjunction is a Quillen equivalence.
\end{prop}

\begin{proof}
First, we easily verify that the left adjoint sends $\tilde{A}\to P$ to the bottom map in the following pushout square.
\[
\xymatrix{
\tilde{A}\ar[r]\ar[d]& P\ar[d]\\
A\ar[r]& Q
}
\]
The result is then completely standard once we know that the model category $\cat{Op}_{ns}$ is left proper (see \cite[Corollary 1.12]{murohomotopy}).
\end{proof}

Denote by $E_1^{ns}$ the non-symmetric little $1$-disks operad (i.e. the operad whose symmetrization is the little $1$-disks operad). Denote by $E_d$ the non-symmetric operad underlying the symmetric little $d$-disks operad. Fix a linear inclusion $\RR\to \RR^d$. This induces a map of non-symmetric operads $E_1^{ns}\to E_d$. The unique map $w:E_1^{ns}\to A$ being a weak equivalence, we can apply the derived functor of $w_!$ (using the notation of Proposition \ref{prop: rigidification}) to it and get a map $A\to P$ for some non-symmetric operad $P$. Moreover, by Proposition \ref{prop: rigidification} this map is weakly equivalent to $E_1^{ns}\to E_d$ in the arrow category of $\cat{Op}^{ns}$. The cosimplicial space associated to this map is what we will denote by $K_d^\bullet$. It is identified with the cosimplicial space mentioned in the beginning of the section.

We can do the exact same construction but starting from the map $L_pE_1\to L_pE_d$ (noticing that $L_pE_1\simeq E_1$) and we obtain a cosimplicial space that we denote by $L_pK_d^\bullet$. Since the map $L_p E_1 \to L_p E_d$ is $GT_p$-equivariant (Theorem \ref{thm:GtonLpE}), the cosimplicial space $L_pK_d^\bullet$ gets a $GT_p$-action.

Applying the functor $C_*(-,R)$ to the cosimplicial space $K^\bullet$, we get a bicomplex and hence a spectral sequence $E^*_{*,*}(R)$ with
\[E^1_{p,q}(R)=H_q(K_d^{-p},R)\]
and whose $d_1$ is the alternating sum of the coface maps. Our main theorem about this spectral sequence is the following (Theorem D in the introduction).

\begin{thm}\label{thm: homology}
Let $p$ be a prime number. The only possibly non-trivial differentials in the spectral sequence $E^*(\ZZ_{(p)})$ are $d_{1+n(d-1)(p-1)}$ for $n \geq 0$.
\end{thm}

\begin{proof}
As in the proof of Theorem B, it suffices to prove the theorem for the spectral sequence $E^*(\ZZ_p)$.  As we said above, we have a $GT_p$-action on the cosimplical space $L_pK^\bullet_d$. Using Proposition \ref{prop: p adics}, we obtain an action on the cosimplicial chain complex $C_*(K_d^\bullet, \ZZ_p)$ and hence on the spectral sequence $E^*(\ZZ_p)$.  Moreover, the action on $E^*_{*,(d-1)k}$ is cyclotomic of weight $k$. Indeed, this is true on the $E^1$ page by Proposition \ref{prop: weight} and therefore on any page by Proposition \ref{prop: submodule and quotient}. The result now follows from Proposition \ref{prop: no map}.
\end{proof}

In particular, we recover the following theorem.

\begin{cor}[Lambrechts-Turchin-Voli\'{c}, \cite{lambrechtsrational}]
The spectral sequence $E^*(\QQ)$ collapses at the second page.
\end{cor}

\begin{proof}
The proof is analogous to the proof of Corollary \ref{cor: collapse over Q}.
\end{proof}

\begin{cor}
For $d \geq 4$ and $p$ a prime, there is an isomorphism
\[
H_i(\overline{\emb}_c(\RR,\RR^d); \ZZ_{(p)}) \cong \bigoplus_{t-s = i} E^2_{-s,t}\otimes\ZZ_{(p)}
\]
for $i \leq (d-3)\left[(p-1)(d-2)+1\right]$.
\end{cor}
\begin{proof}
The  $E^2$ page is zero outside the region bounded by two lines, $t = s(d-1)/2$ (denoted $L$) and $t =  (s-1) (d-1)$ (denoted $U$). The region above the upper line $U$ is already zero on the $E^1$-page since the top non-zero homology of $\emb(\underline{s},\RR^d)$ is in degree $(s-1)(d-1)$. A proof of the lower vanishing line can be found in \cite[Corollary 7.7]{sinhatopology}.

By Theorem \ref{thm: homology}, the first possibly non-zero differential is $d_{1 + (d-1)(p-1)}$. Therefore, the longest differential $d_{r}$ possibly hitting the line $t = i + s$ has source on the line $U$ and target on the line $L$. So under the inequality $r < 1 + (d-1)(p-1)$ the spectral sequence collapses along the line $t = i + s$. An easy calculation shows that this inequality is satisfied if 
\[i \leq (d-3)\left[(p-1)(d-2)+1\right].\]
\end{proof}

\bibliographystyle{acm}

\bibliography{biblio}

\begin{thebibliography}{10}

\bibitem{ALTV}
{\sc Arone, G., Lambrechts, P., Turchin, V., and Voli{\'c}, I.}
\newblock Coformality and rational homotopy groups of spaces of long knots.
\newblock {\em Mathematical Research Letters 15}, 1 (2008), 1--15.

\bibitem{BarNatan}
{\sc Bar-Natan, D.}
\newblock On the {V}assiliev {k}not {i}nvariants.
\newblock {\em Topology 34\/} (1995), 423--472.

\bibitem{BHformality}
{\sc Boavida~de Brito, P., and Horel, G.}
\newblock On the formality of the little disks operad in positive
  characteristic.
\newblock {\em arXiv e-prints\/} (2019), arXiv:1903.09191.

\bibitem{BoavidaWeiss}
{\sc Boavida~de Brito, P., and Weiss, M.}
\newblock Spaces of smooth embeddings and configuration categories.
\newblock {\em Journal of Topology 11}, 1 (2018), 65--143.

\bibitem{BWproduct}
{\sc Boavida~de Brito, P., and Weiss, M.~S.}
\newblock The configuration category of a product.
\newblock {\em Proc. Amer. Math. Soc. 146}, 10 (2018), 4497--4512.

\bibitem{bousfieldlocalization}
{\sc Bousfield, A.~K.}
\newblock The localization of spaces with respect to homology.
\newblock {\em Topology 14\/} (1975), 133--150.

\bibitem{BousfieldKan}
{\sc Bousfield, A.~K., and Kan, D.~M.}
\newblock {\em Homotopy limits, completions and localizations}.
\newblock Lecture Notes in Mathematics, Vol. 304. Springer-Verlag, Berlin-New
  York, 1972.

\bibitem{BCKS}
{\sc Budney, R., Conant, J., Koytcheff, R., and Sinha, D.}
\newblock Embedding calculus knot invariants are of finite type.
\newblock {\em Algebraic and Geometric Topology 17}, 3 (2017), 1701--1742.

\bibitem{BCSS}
{\sc Budney, R., Conant, J., Scannell, K.~P., and Sinha, D.}
\newblock New perspectives on self-linking.
\newblock {\em Advances in Mathematics 191}, 1 (2005), 78 -- 113.

\bibitem{conanthomotopy}
{\sc Conant, J.}
\newblock Homotopy approximations to the space of knots, {F}eynman diagrams,
  and a conjecture of {S}cannell and {S}inha.
\newblock {\em American {J}ournal of {M}athematics 130}, 2 (2008), 341--357.

\bibitem{CTgrope}
{\sc Conant, J., and Teichner, P.}
\newblock Grope cobordism and {F}eynman diagrams.
\newblock {\em Math. Ann. 328}, 1-2 (2004), 135--171.

\bibitem{Drinfeldquasi}
{\sc Drinfel'd, V.~G.}
\newblock On quasitriangular quasi-{H}opf algebras and on a group that is
  closely connected with {${\rm Gal}(\overline{\bf Q}/{\bf Q})$}.
\newblock {\em Algebra i Analiz 2}, 4 (1990), 149--181.

\bibitem{dwyerhess}
{\sc Dwyer, W., and Hess, K.}
\newblock Long knots and maps between operads.
\newblock {\em Geometry \& Topology 16}, 2 (2012), 919--955.

\bibitem{Furusho}
{\sc Furusho, H.}
\newblock Galois action on knots {I}: Action of the absolute {G}alois group.
\newblock {\em Quantum Topology 8}, 2 (2017), 295--360.

\bibitem{goodwilliemultiple}
{\sc Goodwillie, T.~G., and Klein, J.~R.}
\newblock Multiple disjunction for spaces of smooth embeddings.
\newblock {\em Journal of Topology 8}, 3 (2015), 651--674.

\bibitem{GoodwillieWeiss}
{\sc Goodwillie, T.~G., and Weiss, M.}
\newblock Embeddings from the point of view of immersion theory. {II}.
\newblock {\em Geom. Topol. 3\/} (1999), 103--118.

\bibitem{goppl}
{\sc G{\"o}ppl, F.}
\newblock {A spectral sequence for spaces of maps between operads}.
\newblock {\em arXiv e-prints\/} (Oct 2018), arXiv:1810.05589.

\bibitem{Gusarov}
{\sc Gusarov, M.}
\newblock On {$n$}-equivalence of knots and invariants of finite degree.
\newblock In {\em Topology of manifolds and varieties}, vol.~18 of {\em Adv.
  Soviet Math.} Amer. Math. Soc., Providence, RI, 1994, pp.~173--192.

\bibitem{HorelSMF}
{\sc Horel, G.}
\newblock Groupe de galois et espace des noeuds.
\newblock In {\em SMF 2018 : {C}ongr\`es de la {SMF}}, vol.~33 of {\em
  S\'{e}min. Congr.} Soc. Math. France, Paris, 2019, pp.~273--282.

\bibitem{KTchord}
{\sc Kassel, C., and Turaev, V.}
\newblock Chord diagram invariants of tangles and graphs.
\newblock {\em Duke Math. J. 92}, 3 (1998), 497--552.

\bibitem{KST}
{\sc Kosanovi{\'c}, D., Shi, Y., and Teichner, P.}
\newblock Space of gropes and the embedding calculus.
\newblock {\em In preparation\/}.

\bibitem{Kosanovic}
{\sc Kosanović, D.}
\newblock Embedding calculus and grope cobordism of knots.
\newblock {\em arXiv e-prints\/} (2020), arXiv:2010.05120.

\bibitem{lambrechtsrational}
{\sc Lambrechts, P., Turchin, V., and Voli{\'c}, I.}
\newblock The rational homology of spaces of long knots in codimension $>$ 2.
\newblock {\em Geometry \& Topology 14}, 4 (2010), 2151--2187.

\bibitem{murohomotopy}
{\sc Muro, F.}
\newblock Homotopy theory of non-symmetric operads, {II}: Change of base
  category and left properness.
\newblock {\em Algebraic \& Geometric Topology 14}, 1 (2014), 229--281.

\bibitem{Shi}
{\sc Shi, Y.}
\newblock Goodwillie's cosimplicial model for the space of long knots and its
  applications.
\newblock {\em arXiv e-prints\/} (2020), arXiv:2012.04036.

\bibitem{sinhaoperads}
{\sc Sinha, D.}
\newblock Operads and knot spaces.
\newblock {\em Journal of the American Mathematical Society 19}, 2 (2006),
  461--486.

\bibitem{sinhatopology}
{\sc Sinha, D.~P.}
\newblock The topology of spaces of knots: cosimplicial models.
\newblock {\em American journal of mathematics 131}, 4 (2009), 945--980.

\bibitem{turchindelooping}
{\sc Turchin, V.}
\newblock Delooping totalization of a multiplicative operad.
\newblock {\em J. Homotopy Relat. Struct. 9}, 2 (2014), 349--418.

\bibitem{Vassiliev}
{\sc {Vasil'ev}, V.~A.}
\newblock {\em {Complements of discriminants of smooth maps: topology and
  applications. Transl. from the Russian by B. Goldfarb. Transl. ed. by S.
  Gelfand. Rev. ed.}}, vol.~98.
\newblock Providence, RI: American Mathematical Society, 1994.

\bibitem{volic}
{\sc Volić, I.}
\newblock Finite type knot invariants and the calculus of functors.
\newblock {\em Compositio Mathematica 142}, 01 (2006), 222--250.

\end{thebibliography}

\end{document}